\documentclass[11pt,a4paper]{amsart}
\usepackage{latexsym,amssymb}
\usepackage{pgf}

\begin{document}
\title[Strong minimal covers and hyperimmune-free
degrees]{$ \Pi^0_1 $ classes, strong minimal covers and
hyperimmune-free degrees}
\author{Andrew E.M. Lewis}

\def\electronicmail#1{{\em Email:} {\tt #1}\,. \mbox{ }}
\def\alladdress{\noindent
 {\bf Andrew E.M.~Lewis.} Dipartimento di Scienze Matematiche ed
Informatiche, Pian dei Mantellini 44, 53100 Siena, Italy;
\electronicmail{andy@aemlewis.co.uk}}

\thanks{2000MSC: 03D28. The author was supported by Marie-Curie Fellowship MEIF-CT-2005-023657 and was
partially supported by the NSFC Grand International Joint Project no. 60310213, New Directions in the Theory and Applications of Models of Computation. }

\maketitle

\begin{abstract} We investigate issues surrounding an old question of Yates' as to the existence of a
minimal degree with no strong minimal cover, specifically with
respect to the hyperimmune-free degrees.
\end{abstract}

\newtheorem{theo}{Theorem}[section]
\newtheorem{defin}{Definition}[section]
\newtheorem{lem}{Lemma}[section]
\newtheorem{ques}{Question}[section]
\newtheorem{coro}{Corollary}[section]
\newtheorem{prop}{Proposition}[section]
\section{Introduction}

If $A$ and $B$ are sets of natural numbers we write $A\leq_T B$ if
$A$ is Turing reducible to $B$---intuitively, $A$ can be computed if
we are able to compute $B$. The Turing reducibility induces an
equivalence relation on the sets of natural numbers and an order $<$
on the equivalence classes. These equivalence classes are called the
Turing degrees. Intuitively each degree represents an information
content, since all sets in the same degree code the same
information. An old question of Yates' remains one of the
longstanding problems of degree theory:

\begin{defin} We write $ \boldsymbol{0} $ to denote the least Turing degree. For any degree $\boldsymbol{a}$  we write
$\mathcal{D}[<\textbf{a}]$ in order to denote the set of degrees
strictly below $\boldsymbol{a}$. A degree $\boldsymbol{a}$ is
\textbf{minimal} if $\mathcal{D}[<\textbf{a}]=\{ \boldsymbol{0} \}$.
A degree $ \bf{b} $ is a \textbf{strong minimal cover} for $
\textbf{a} $ if $ \mathcal{D}[<\textbf{b}] =\mathcal{D}[\leq
\textbf{a}] $.
\end{defin}

\begin{ques}[Yates] Does every minimal degree have a strong minimal cover?
\end{ques}

It seems fair to say that for a long time very little progress was
made in the attempt to understand issues surrounding Yates' problem.
This situation changed relatively recently, however, with
Ishmukhametov's characterization of the computably enumerable (c.e.)
degrees which have a strong minimal cover. The first ingredient here
was provided Downey, Jockusch and Stob \cite{DJS} in 1996. They
defined a degree $ \bf{a} $ to be \textbf{array nonrecursive} (or
\textbf{array incomputable}) if for each $ f\leq_{wtt} K $ there is
a function $ g $ computable in $ \bf{a} $ such that $ g(n) \geq f(n)
$ for infinitely many $ n $, where $K$ denotes Turing's halting
problem, and $f\leq_{wtt} g$ if $ f\leq_T g $ and there is a
computable bound on the number of arguments of $g $ which are
required on any given input.

\begin{theo}[Downey,Jockusch,Stob \cite{DJS}] Given $ \bf{a} $ which is a.i.c.:
\begin{enumerate}
\item $ \bf{a} $ is not minimal,

\item if $ \bf{c}>\bf{a} $ then there is a degree $ \bf{b}< \bf{c} $ such that $ \bf{a}\vee \bf{b}=\bf{c} $.
\end{enumerate}
\end{theo}

\noindent Ishmukhametov then combined this work to great effect with
an analysis of the  \textbf{c.e. traceable} degrees (referred to by
him as  \textbf{weakly recursive}):

\begin{defin} \label{cet}
$ A \subseteq \omega $ is \textbf{c.e. traceable} if there is a
computable function $ p $ such that for every function $ f\leq_T A $
there is a computable function $ h $ such that $ \vline W_{h(n)}
\vline \leq p(n) $ and $ f(n)\in W_{h(n)} $ for all $ n\in \omega $.
\end{defin}

\noindent Having observed that the class of c.e. traceable degrees
(those whose sets are c.e. traceable) is complementary to the class
 of a.i.c. degrees in the c.e. degrees, Ishmukhametov \cite{SI} was then able to show that all c.e. traceable degrees
have a strong minimal cover.

\begin{theo}[Ishmukhametov \cite{SI}] A c.e.  degree has a strong minimal cover iff it is c.e. traceable.
\end{theo}

\begin{defin} We say a degree $ \bf{a} $ satisfies the
\textbf{cupping property} if for every $ \bf{c}>\bf{a} $ there
exists $ \bf{b}<\bf{c} $ with $ \bf{a}\vee \bf{b}=\bf{c} $.
\end{defin}

It is worth remarking that, as an immediate consequence of these
results, every c.e. degree either has a strong minimal cover or
satisfies the cupping property. It is an open question as to whether
this is true of the Turing degrees in general.

\begin{defin}
Given any $ T \subseteq 2^{<\omega} $  and $ \tau,\tau' \in T $, we
say that $ \tau $ is a \textbf{leaf} of $ T $ if it has no proper
extensions in $ T $. When $ \tau \subset \tau' $ we call $ \tau' $ a
\textbf{successor} of $ \tau $ in $ T $ if there doesn't exist $
\tau'' \in T $ with $ \tau \subset \tau'' \subset \tau' $.
\end{defin}

\begin{defin}Given any $ T \subseteq 2^{<\omega} $  and $ A\subseteq
\omega $ we denote $ A\in [T] $ if there exist infinitely many
initial segments of $ A $ in $ T $.
\end{defin}

The following definition is slightly non-standard, but seems
convenient where the discussion of splitting trees with unbounded
branching is concerned:

\begin{defin} \label{deftree}
We say that $ T \subseteq 2^{<\omega} $  is a \textbf{c.e. tree} if
it has a computable enumeration $ \{ T_s \}_{s\geq 0} $ such that $
\vline T_0 \vline = 1 $ and such that for all $ s\geq 0 $ if $ \tau
\in T_{s+1}-T_s $ then $ \tau $ extends a leaf of $ T_{s} $ (and
such that a finite number of strings are enumerated at any stage).
\end{defin}

\begin{defin}
We say that $ T \subseteq 2^{<\omega} $  has \textbf{bounded
branching} if there exists $ n $ such that every $ \tau\in T $ has
at most $ n $ successors in $ T $.
\end{defin}

\begin{defin}
We say that $ T $ is \textbf{$ \Psi $-splitting} if whenever $
\tau,\tau'\in T $ are incompatible, $ \Psi(\tau) $ and $ \Psi(\tau')
$ are incompatible. We say that $ A $ satisfies the (bounded
branching) \textbf{splitting tree property} if whenever $ A \leq_T
\Psi(A) $, $ A $ lies on a c.e. $ \Psi $-splitting tree (with
bounded branching).
\end{defin}

\begin{defin} For any $ \tau \in 2^{<\omega} $ if $ \vline \tau \vline >0 $
we define $ \tau^{-} $ to be the initial segment of $ \tau $ of
length $ \vline \tau \vline -1 $, and otherwise we define $
\tau^-=\tau$ . Given any Turing functional $ \Psi $ we define $ \hat
\Psi $ as follows. For all $ \tau $ and all $ n $, $ \hat
\Psi(\tau;n)\downarrow =x $ iff the computation $ \Psi(\tau;n) $
converges in $ < \vline \tau \vline $ steps, $ \Psi(\tau;n)=x $ and
$ \hat \Psi(\tau^{-};n')\downarrow $ for all $ n'<n $.
\end{defin}

\begin{defin}
We say that $ \tau \in T $ is of \textbf{level n} in $ T $ if there
exist precisely $ n $ proper initial segments of $ \tau $ in $ T $.
We say that $ T $ is of \textbf{level at least  n} if all leaves of
$ T $ are of level at least $ n $ in $ T $. We say that $ T $ is of
\textbf{level  n} if all leaves of $ T $ are of level $ n $ in $ T
$.
\end{defin}

Unfortunately c.e. traceability does not relate in such a tidy way
where the minimal degrees are concerned. Gabay \cite{YG} has shown
that there are minimal degrees with strong minimal cover and which
are not c.e. traceable. Since the set that Gabay constructs in this
proof satisfies the splitting tree property and since any set
satisfying the bounded branching splitting tree property is c.e.
traceable, it follows that there exist minimal degrees containing a
set which satisfies the splitting tree property which do not contain
a set satisfying the bounded branching
 splitting tree property. In order to see that any set $A$
satisfying the bounded branching splitting tree property is c.e.
traceable, let $ p $ dominate all functions of the form $ g(n)=m^n$.
If $ A $ lies on a c.e. $\hat \Psi$-splitting tree in which each
string has at most $ m $ successors, then $\Psi(A;n) $ is contained
in the set $ \{ \hat \Psi(\tau;n): \tau \ \mbox{is of level}\ n+1\}$
and this set is of size at most $m^n$. It is therefore clear how we
may define $h$ so as to satisfy the definition of c.e. traceability.

\vspace{4pt} It is the aim of this paper to further our
understanding of the issues surrounding Yates' question, in
particular where the hyperimmune-free degrees are concerned.

\begin{defin} $ A\subseteq \omega $ is \textbf{hyperimmune-free}
if for every $ f\leq_T A $ there exists a computable function $ h $
which dominates $ f $ (or equivalently which majorizes $ f $) i.e.
such that $ h(n)\geq f(n) $ for all but finitely many $ n $. In his
new book Soare will introduce  the terminology $ \bf{0}
$-\textbf{dominated} in place of hyperimmune-free. Preferring this
terminology, we shall adopt it in what follows.
\end{defin}

\begin{defin} \label{defper} We say non-empty  $ T \subseteq 2^{<\omega} $ is \textbf{perfect} if
 each $ \tau \in T $ has at least two  successors in $ T $.
\end{defin}

In some ways the  $ \bf{0} $-dominated degrees and the minimal
degrees may be regarded as quite intimately related---the standard
constructions, at least, are very similar and provide many of the
same restrictions. The most basic form of minimal degree
construction produces a set which satisfies the perfect splitting
tree property and it may be observed that every such set is, in
fact, of $ \bf{0} $-dominated degree. In order to see this we can
argue as follows. Given  $ f $ computable in $ A $ and which is
incomputable, let $ \Psi(A)=f $. If $ A $ lies on a perfect c.e. $
\hat \Psi $-splitting tree then for every $ n $, $ f(n) $ is
included in the values $ \hat \Psi(\tau;n) $ for those $ \tau $ of
level $ n+1 $ in this tree.

\vspace{4pt}
 In what follows all notations and terminologies will either be standard or
explicitly defined. For an introduction to the techniques of minimal
degree construction we refer the reader to any one of \cite{RS},
\cite{BC2}, \cite{ML}, \cite{PO1} (this paper requires knowledge
only of Spector forcing).

\vspace{4pt} The author would like to thank Barry Cooper, Richard
Shore, Frank Stephan, Bj\o rn Kjos-Hanssen, Carl Jockusch, Andrea
Sorbi, Jan Reimann and Theodore Slaman for helpful discussions.

\section{Bushy trees}
Since the splitting tree technique is very much the standard
approach to minimal degree construction it is natural to ask the
generality of this method. Given that it seems, in the very least,
to be difficult to construct a minimal degree with no strong minimal
cover using the splitting tree technique it is an obvious question
to ask whether we should be looking for alternative methods of
minimal degree construction. The simple  observation of this section
is that the sets of minimal degree can, in fact, be completely
characterized in terms of \textbf{weakly splitting trees}, or in
terms of c.e. trees with \textbf{delayed splitting}. These kinds of
splitting trees, then, provide a perfectly general method of minimal
degree construction. While it remains to be shown that their use
will really enable one to do anything that the use of standard
splitting trees will not, we shall argue later that the use of these
kinds of trees  may be necessary in constructing a negative solution
to Yates' question---or, at least, in constructing such a degree
which is also $ \bf{0} $-dominated. The result corresponding
 to theorem \ref{wst} for the sets of degree below $ \bf{0}' $
 appeared in Odifreddi \cite{PO2}, who directly transposed
 techniques previously described
 by Chong \cite{CC} while working in $ \alpha $-recursion theory.

\begin{defin}
We say that $ T \subseteq 2^{<\omega} $ is a \textbf{weakly $ \Psi
$-splitting tree} if it is a c.e. subset of $ 2^{<\omega } $ and
there exist partial computable $ \phi,\psi:2^{<\omega}\rightarrow
\omega $ such that:

\begin{enumerate}
\item if $ \tau \in T $ then $\phi(\tau)\downarrow < \vline \tau \vline $ and
$ \psi(\tau)\downarrow < \vline \Psi(\tau) \vline $,

\item whenever $ \tau,\tau' \in  T $ are incompatible below min$ \{ \phi(\tau),\phi(\tau') \} $
they $ \Psi $-split below min$ \{ \psi(\tau),\psi(\tau') \} $,

\item if $ A\in [ T ]$ then the set $ \{ \phi(\tau): \tau \subset A, \tau \in T \} $ has no finite upper bound.
\end{enumerate}
\end{defin}

\begin{lem} \label{lowst}
For any $ A \subseteq \omega $ and any Turing functional $ \Psi $,
if $ \Psi(A) $ is total and $ A \leq_T \Psi(A) $ then $ A $ lies on
a weakly $ \Psi $-splitting tree.
\end{lem}
\begin{proof}
Suppose $ \Phi(\Psi(A))=A $. In order to enumerate $ T $ run through
all computations of the form $ \hat \Phi(\hat \Psi(\tau)) $ for all
$ \tau \in 2^{<\omega} $, one string at a time and in order of
length. Whenever we find $ \tau,n $ such that $ \hat \Phi(\hat
\Psi(\tau)) $ is defined and agrees with $ \tau $ on all arguments
 $ \leq n $ and such that this is not the case for any $ \tau' \subset \tau $
enumerate $ \tau $ into $ T $ and define $ \phi(\tau)=n $ and $
\psi(\tau)=m $, where $ m $ is the maximum $ x $ such that $ \hat
\Psi(\tau;x) $ is used in computations $  \hat \Phi(\hat
\Psi(\tau);n') $ for $ n'\leq n $.
\end{proof}

\begin{lem}
If $ A $ lies on $ T $ which is a weakly $ \Psi $-splitting tree and
$ \Psi(A) $ is total then $ A\leq_T \Psi(A) $.
\end{lem}
\begin{proof}
Suppose we are given an oracle for $ \Psi(A) $. In order to compute
the initial segment of $ A $ of length $ n $
 proceed as follows. Enumerate $ T $ until we find
$ \tau \in T $ such that $ \Psi(\tau) $ agrees with $ \Psi(A) $ on
all arguments $ \leq \psi(\tau) $, and such that $ \phi(\tau) \geq
n-1 $. The initial segment of $ \tau $ of length $ n $ is an initial
segment of $ A $.
\end{proof}

We therefore have:

\begin{theo} \label{wst}
$ A $ is of minimal degree iff $ A $ is incomputable and,
 whenever $ \Psi(A) $ is total and incomputable, $ A $ lies on a weakly $ \Psi $-splitting tree.
\end{theo}

Theorem \ref{delay} below gives a version of theorem \ref{wst} in
more familiar form. In order to get a perfectly general method of
minimal degree construction all we need do is delay splitting by one
level:

\begin{defin} We say that $ T \subseteq 2^{<\omega} $ is a delayed $
\Psi $-splitting c.e. tree if $ T $ is a c.e. tree and whenever $
\tau_0, \tau_1 \in T $ are incompatible (denoted $ \tau_0 \vline
\tau_1$) any $ \tau_2,\tau_3 \in T $ properly extending $ \tau_0 $
and $ \tau_1 $ respectively are a $  \Psi $-splitting.
\end{defin}

\begin{theo} \label{delay}
$ A $ is of minimal degree iff $ A $ is incomputable and,
 whenever $ \Psi(A) $ is total and incomputable, $ A $ lies on a delayed $ \Psi $-splitting c.e. tree.
\end{theo}

\begin{proof} (Sketch) The direction from right to left is easy (note that, if
 whenever $ \Psi(A) $ is total and incomputable $ A $ lies on a delayed $ \Psi $-splitting c.e. tree, then
 $ A $ also lies on a delayed $ \hat \Psi $-splitting c.e. tree) so
we are left to prove the direction from left to right. So suppose
given $ A \subseteq \omega $ such that  $ \Psi(A) $ is total and $ A
\leq_T \Psi(A) $. Let $ T $ be the weakly splitting tree as defined
in the proof of lemma \ref{lowst}, but with also the empty string as
the single member of level 0. First we define $ T' $ as follows. The
strings of level $ n\leq 1 $ in $ T' $ are the strings of level $
n\leq 1 $ in $ T $. Suppose we have already defined the strings of
level $ n\geq 1 $ in $ T' $. For each $ \tau \in T' $ of level $ n $
in $ T' $ the strings extending $ \tau $ in $ T' $ of level $ n+1 $
in $ T' $ are those strings extending $ \tau $ in $ T $ and which
are of level $ \vline \tau \vline  $ in $ T $. It remains to show
that we can take $ T'' $ such that for any $ n $ the strings of
level $ n $ in $ T'' $ are strings of level $ n $ in $ T' $, such
that $ A\in [T''] $ and which is a c.e. tree. In order to see this
consider, during the enumeration of $ T' $,  the enumeration of
axioms for $ \Phi $ which attempts to map every set lying on $ T' $
to a 1-generic. Since $ A $ is if minimal degree and therefore does
not bound a 1-generic there will exist a least $ i $ such that for
all $ \tau \subset A $, $ \Psi_i(\Phi(\tau);i)\uparrow $ and there
exists $ \sigma \supset \Phi(\tau) $ with $
\Psi_i(\sigma;i)\downarrow $. So long as the axioms for $ \Phi $ are
enumerated in a reasonably intelligent way we may then define $ T''
$ by enumerating strings of level $ n+1 $ (for sufficiently large $
n $) into this tree extending $ \tau $ of level $ n $ only when we
find $ \sigma \supset \Phi(\tau) $ with $ \Psi_i(\sigma;i)\downarrow
$, and by insisting that no successors of $ \tau $ should be
enumerated at any subsequent stage.
\end{proof}

\section{$ \Pi^0_1 $ classes and the cupping property}

\begin{defin} For any $ T \subseteq \omega^{<\omega} $ we define $
D(T)$ (the \textbf{downwards closure} of $ T $)  to be the set of
all $ \tau $ for which there exists $ \tau' \supseteq \tau $ in $
T$. We say that $ T\subseteq \omega^{<\omega} $ is \textbf{downwards
closed} if $ D(T)=T $.
\end{defin}

\begin{defin} Downwards closed computable $ T \subseteq \omega^{<\omega} $
is said to be \textbf{highly computable} if there exists a partial
computable function $ f $ such that, for any $ \tau \in T $, $ \tau
$ has at most $ f(\tau) $ successors in $ T $. A subset $ X $ of $
\omega^{\omega} $ is a $ \Pi_{1}^{0} $ \textbf{class} if $ X=[T] $
for some computable downwards closed $ T $, and if $ T $ is highly
computable then $ X $ is said to be a \textbf{computably bounded
(c.b.)} $ \Pi_{1}^{0} $ \textbf{class}.
\end{defin}

In what follows, it may be assumed that any $ \Pi^0_1 $ class
referred to is computably bounded.

\begin{theo}[Jockusch and Soare] The $ \bf{0} $-dominated basis theorem: every non-empty  $ \Pi_1^0 $
class contains a member of $ \bf{0} $-dominated degree.
\end{theo}

\begin{defin} A degree $ \bf{a} $ is PA if it is the degree of a
complete extension of Peano Arithmetic. Equivalently, $ \bf{a} $ is
PA if every non-empty $ \Pi^0_1 $ class contains a member of degree
$ \leq \bf{a} $.
\end{defin}

\vspace{4pt} Given the similarity between the standard constructions
of $ \bf{0} $-dominated and minimal degrees, perhaps the strongest
result on the negative side of Yates' question is that there exists
a $ \bf{0} $-dominated degree which satisfies the cupping property
(and so does not have a strong minimal cover). This follows directly
from the theorem of Kucera's that the PA degrees satisfy the cupping
property using the $ \bf{0} $-dominated basis theorem, since there
exists a non-empty $ \Pi^0_1 $ class every member of which is of PA
degree.

\begin{theo}[Kucera \cite{AK}] \label{nsm} The PA degrees satisfy the cupping property.
Equivalently, there exists a non-empty  $ \Pi_1^0 $ class every
member of which is of degree which satisfies the cupping property.
\end{theo}

We give here a simple proof of theorem \ref{nsm}, the hope being
that this alternative proof (which doesn't require reasoning within
PA like the original) may be more flexibly extended in order to give
stronger results. The point of this proof is not so much that it is
the shortest possible, but rather that it very effectively exposes
the combinatorial arguments which lie at the heart of issues
surrounding the cupping property and the construction of strong
minimal covers. This proof has already provided the intuition behind
the constructions appearing in \cite{AL1} and \cite{AL2} and can be
expected to have further applications.

\begin{defin} We let $ \lambda $ denote the string of length 0.
 $ T \subseteq 2^{<\omega} $ with a single element of level 0
 is  \textbf{2-branching} if every $ \tau\in T $ has two precisely
 two successors in $ T $.
We say that $ T $ is \textbf{2-branching below level  n} if all $
\tau \in T $ which are of level $ n'<n $ in $ T $ have precisely two
successors in $ T $.
\end{defin}

\noindent \emph{Alternative proof of theorem \ref{nsm}}.  We observe
first that if  $ A \subseteq 2^{\omega} $ satisfies the property
that there exists a 2-branching $ T\leq_T A $ such that if $ C\in
[T] $ then $ A\not \leq_T C $, then  $ deg(A) $ satisfies the
cupping property. Given $ B $ of degree strictly above $ A $ we
define $ C<_T B $ such that $ C= \bigcup_n \sigma_n $. Let $
\sigma_0$ be the string of level 0 in $ T $ and for all $ n>0 $ let
$ \sigma_n $ be the right successor of $ \sigma_{n-1} $ in $ T $ if
$ B(n-1)=1 $ and let $ \sigma_n $ be the left successor otherwise.
Then $ B\not \leq_T C $ since $ C $ lies on $ T $  and it is clear
that $ B $ is computable given oracles for $ A $ and $ C $.

\vspace{4pt} In order to construct a non-empty  $ \Pi_1^0 $ class
every member of which is of degree which satisfies the cupping
property, then, it suffices to construct downwards closed computable
$ \Pi \subseteq 2^{<\omega} $ such that $ [\Pi] $ is non-empty and
such that for every $ A\in [\Pi] $ there exists 2-branching $
T^A\leq_T A $ which satisfies the property that if $ C\in [T^A] $
then $ A\not \leq_T C $. In order to ensure that $ A\not \leq_T C $
for any $ C\in [T^A] $, we shall construct $ \Psi $ such that $
\Psi(A) \not \leq_T C $. In particular we shall construct $ \Pi $
and $ \Psi $ so that $ [\Pi] $ is non-empty and so as to satisfy
every requirement:

\[ \mathcal{N}_i: (A\in [\Pi] \wedge C\in [T^A])\rightarrow ( \Psi_i(C;i)\neq
\Psi(A;i)). \]

\vspace{4pt} \noindent  For $ \tau $ in $ \Pi $ we shall define
values $ T^{\tau} $. If $ A $ is an infinite path through $ \Pi $
then $ T^A $ will be defined to be $ \bigcup \{ T^{\tau}:
T^{\tau}\downarrow, \tau \subset A \} $.

\vspace{8pt} So let us consider first how to satisfy a single
requirement $ \mathcal{N}_0 $. We wish to construct $ \Pi $ such
that $ [\Pi] $ is non-empty and if $ A\in [\Pi] $ there exists
2-branching $ T^A\leq_T A $ which satisfies the property that if $
C\in [T^A] $ then $ \Psi_0(C;0)\neq \Psi(A;0) $. The most primitive
intuition here is as follows; if we are given four strings and we
colour these strings with two colours then there exists some colour
such that at least two strings are not that colour (okay so actually
we only need three strings, but it convenient here to do everything
in powers of two).

Now we extend this idea. Let $ T $ be the set of finite binary
strings which are of even length, the important point here being
that $ T $ is 4-branching. We let $ T(n) $ denote the strings in $ T
$ of level $ n $ in $ T $.

\begin{defin} For any finite $ T'\subseteq 2^{<\omega} $ and any $ m $, an $
m $-colouring of $ T' $ is an assignment of some col$(\sigma)<m $ to
each leaf $ \sigma $ of $ T' $.
\end{defin}

\begin{defin} Given any $ f:\omega\rightarrow \omega $
 we say that non-empty $ T' $ is $ (T,f) $ \textbf{compatible} if
  for every $ n $ the strings of level $ n $ in $ T' $ are
 strings of level $ n $ in $ T $ and any string of level $ n $ in $
 T' $ which is not a leaf of $ T' $ has $ f(n) $ successors in $
 T' $.
\end{defin}

Let $ \kappa $ be the constant function such that for all $ n $, $
\kappa(n)=2 $. We say that $ T' $ is $ (T,2) $ compatible if it is $
(T,\kappa) $ compatible. The following lemma is just what we need in
order to be able to satisfy the single requirement $ \mathcal{N}_0
$.

\begin{lem} \label{twocol}
For every $ n $ and every 2-colouring of $ T(n) $ there exists $ d<2
$ and $ (T,2) $ compatible $ T' $  of level $ n $ such that no leaf
$ \sigma $ of $ T' $ has col$(\sigma)= d $.
\end{lem}

\begin{proof}
\vspace{4pt} \noindent The case $ n=0 $ is trivial and, in fact, we
have already seen the case $ n=1 $ since there are four strings in $
T(1) $ and any two of these strings define a $ (T,2) $ compatible $
T' $ of level one.

\vspace{4pt} So suppose the result holds for all $ n'\leq n $. Given
 any 2-colouring of $ T(n+1) $ consider each $ \sigma \in T(n) $.
 Each such $ \sigma $ has four successors in $ T $ which are strings
 in
 $  T(n+1) $.
If there exists $ d<2 $ such that more than two of the successors $
\sigma' $ of $ \sigma $ in $ T $ have col$(\sigma')=d$ then define
col$(\sigma)=d $ and otherwise define col$(\sigma)=0$. This gives a
2-colouring of $ T(n) $ and by the induction hypothesis there exists
$ T' $ which is $ (T,2) $ compatible of level $ n $ and $ d<2 $ such
that no leaf $ \sigma $ of $ T' $ has col$(\sigma)=d $. In order to
define $ T'' $ which is $ (T,2) $ compatible and of level $ n+1 $
and such that no leaf $ \sigma' $ of $ T'' $ has col$(\sigma')=d $
just choose two extensions of each leaf $ \sigma $ of $ T' $ which
are not coloured $ d $. We have defined the 2-colouring of $ T(n) $
precisely so that this is possible.
\end{proof}

Now we see how to use this lemma in order to satisfy the first
requirement. Before defining $ \Pi $ we define a set of strings $
\Pi^{\star} $---these are strings which may or may not be in $ \Pi
$. We do not require that $ \Pi^{\star} $ is downwards closed. Once
we have defined this set we shall form $ \Pi $ by taking certain
strings from $ \Pi^{\star} $ and then adding strings in so that $
\Pi $ will be downwards closed. For every $ n $ we have to define
the set $ \Pi^{\star}(n) $ which is the set of strings in $
\Pi^{\star} $ of level $ n $ in $ \Pi^{\star} $, for each of these
strings $ \tau $ we have to define a value $ T^{\tau} $ and we also
have to ensure that $ \Psi(\tau;0)\downarrow $. The latter condition
we satisfy by defining $ \Psi(\tau;0) $ for all $ \tau \in
\Pi^{\star}(1) $. We shall explain exactly how one may define $
\Pi^{\star} $ and the other values just discussed in a moment, but
the important point here is just this; we can do so in such a way
that for any $ d<2 $ and any $ (T,2) $ compatible $ T' $ of level $
n \geq 1 $ there exists $ \tau\in \Pi^{\star}(n) $ such that $
T^{\tau}=T' $ and $ \Psi(\tau;0)= d $. Really this is completely
obvious---all you need do is to put enough strings in $
\Pi^{\star}(n) $ so that all possibilities can be realized.

\vspace{4pt} What this means is that if we define $ \Pi $ by taking
the strings in $ \Pi^{\star} $ except for those strings $ \tau $ for
which we observe that there exists $ \sigma \in T^{\tau} $ with $
\Psi_0(\sigma;0)=\Psi(\tau;0) $ then for every $ n $ there must
exist $ \tau \in \Pi^{\star}(n) $ which is in $ \Pi $ (and thus $
[\Pi] $ will be non-empty). This follows because we can consider the
values $ \Psi_0(\sigma;0) $ for $ \sigma $ in $ T(n) $ to define a
2-colouring of $ T(n) $. For any 2-colouring of $ T(n) $ there
exists $ d<2 $ and a $ (T,2) $ compatible $ T' $ of level $ n $ such
that no leaf of $ T' $ is coloured $ d $. Then $ \tau \in
\Pi^{\star}(n) $ with $ T^{\tau}=T' $ and $ \Psi(\tau;0)=d $ will be
a string in $ \Pi $.

\vspace{4pt} Before going on to consider how we may satisfy all
requirements, then, let's see how to define $ \Pi^{\star} $ (when we
are only looking to satisfy the first requirement). We define $
\Pi^{\star}(n) $ by recursion on $ n $. We define $
\Pi^{\star}(0)=\{ \lambda \} $ and $ T^{\lambda}= \{ \lambda \} $.
 Let $ \{T_0,...,T_{m-1} \} $ be the set of all $
(T,2) $ compatible $ T' $  which are of level $ 1 $. Let $
\tau_0,...,\tau_{2m-1} $ be pairwise incompatible, define $
\Pi^{\star}(1)= \{ \tau_0,...,\tau_{2m-1} \} $ and for each $ i< m $
define $ T^{\tau_{2i}}=T_i $, $ T^{\tau_{2i+1}}=T_i $, $
\Psi(\tau_{2i};0)=0 $ and $ \Psi(\tau_{2i+1};0)=1 $.

 Given $ \Pi^{\star}(n) $ for $ n>0 $ we define $ \Pi^{\star}(n+1)
$ as follows.
 For each  $ \tau \in \Pi^{\star}(n) $ let $
\{T_0,...,T_{m'-1} \} $ be the set of all $ (T,2) $ compatible $ T'
$ such that $ T^{\tau}\subset T' $ and which are of level $ n+1 $.
Let $ \tau_0,...,\tau_{m'-1} $ be pairwise incompatible extensions
of $ \tau $, enumerate these strings into $ \Pi^{\star}(n+1) $ and
for each $ i< m' $ define $ T^{\tau_i}=T_i $.

\vspace{4pt} In order to satisfy every requirement $ \mathcal{N}_i $
while maintaining non-empty  $ [\Pi] $ we must become a little more
sophisticated, but the basic idea remains the same. We need a
`bushier' $ T $ and we need also to use more colours for lower
priority requirements:

\begin{defin}
For every $ n $ the set $ T(n) $---the strings in $ T $ of level $ n
$--- are those elements of $ 2^{<\omega} $ of length $ \Sigma_{i< n}
(i+2) $.
\end{defin}

\begin{defin} For every i we let $ \kappa_i $ be defined as follows.
We have $ \kappa_0(0)=4 $. For all $n<i $, $ \kappa_i(n)=2 $ and for
all $ n\geq i $, $ \kappa_i(n)=2\kappa_i(n-1) $ (if $n-1\geq 0 $).
For all $ i $ we define ncol$(i)=2^{i+1} $.
\end{defin}

For every $ i $, then, the value ncol$(i) $ should be thought of as
the number of colours that we use in order to satisfy the
requirement $ \mathcal{N}_i $. We need a new version of lemma
\ref{twocol}:

\begin{lem} \label{nice} If $ T_0 $ is  $ (T,\kappa_i) $ compatible and of level $ n $
then for any ncol$(i)$-colouring of $ T_0 $ there exists $
T_1\subseteq T_0 $ which is $ (T,\kappa_{i+1}) $ compatible of level
$ n $ and $ d<$ncol$(i) $ such that no leaf $ \sigma $ of $ T_1 $
has col$(\sigma)=d$.
\end{lem}

\begin{proof}
Given any fixed $ i $ we prove the result by induction on $ n $. The
base case, for those $ n\leq i $, follows trivially since there are
at most $ 2^{i} $ strings in $ T_0 $ of level $ n $ and
ncol$(i)=2^{i+1} $. So suppose that $n\geq i $ and that the result
holds for $ n $. Given any ncol$(i)$-colouring of $ T_0 $ which is $
(T,\kappa_i) $ compatible and of level $ n+1 $ consider each string
$ \sigma $ of level $ n $ in $ T_0 $. Such $ \sigma $ has $
2^{n-i+2} $ successors $ \sigma' $ in $ T_0 $. If there exists some
$ d<$ncol$(i)$ such that more than half of those successors $
\sigma' $ have col$(\sigma')=d $ then define col$(\sigma)=d$ and
otherwise define col$(\sigma)=0$. Let $ T_0' $ be the set of strings
in $ T_0 $ of level $\leq n $. We have defined an
ncol$(i)$-colouring of $ T_0' $. By the induction hypothesis there
exists $ T_1'\subseteq T_0' $ which is $ (T,\kappa_{i+1}) $
compatible of level $ n $ and $ d<$ncol$(i) $ such that no leaf $
\sigma $ of $ T_1' $ has col$(\sigma)=d$. In order to define $
T_1\subseteq T_0 $ which is $ (T,\kappa_{i+1}) $ compatible of level
$ n+1 $  such that no leaf $ \sigma $ of $ T_1 $ has
col$(\sigma)=d$, simply choose $ 2^{n-i+1} $ extensions $ \sigma' $
of each leaf $ \sigma $ of $ T_1' $ such that col$(\sigma')\neq d $.
We defined the colouring of $ T_0' $ precisely so that this is
possible.
\end{proof}

The intuition now runs as follows. For each $ i $ we shall define
any value $ \Psi(\tau;i)=d $ so that $ d<\mbox{ncol}(i)  $. We are
yet to define $ \Pi $ and $ \Pi^{\star} $, but we will shortly do so
in a manner analogous to what went before. Suppose that for some $ n
$ there are no members of $ \Pi^{\star}(n) $ in $ \Pi $. Let $
T_0=\bigcup_{i\leq n}T(i) $. Then we consider those values $ \hat
\Psi_0(\sigma;0) $ for the leaves $ \sigma $ of $ T_0 $ to define a
2-colouring of $ T_0 $. Lemma \ref{nice} then suffices to show, not
only that there exists $ \tau\in \Pi^{\star}(n) $ such that no leaf
$ \sigma $ of $ T^{\tau} $ has col$(\sigma)=\Psi(\tau;0) $, but that
there exists a range of such values---all those $ \tau $, in fact,
such that a) $ T^{\tau} $ is a subset of some fixed $ T_1\subset T_0
$ which is $ (T,\kappa_1) $ compatible of level $ n $ and b) such
that $ \Psi(\tau;0)=d_0 $ for some fixed $ d_0<2 $. Next we consider
those values $ \hat \Psi_1(\sigma;1) $ for the leaves $ \sigma $ of
$ T_1 $ to define a 4-colouring of $ T_1 $. Once again we apply
lemma \ref{nice} in order to obtain $ T_2\subset T_1 $ which is $
(T,\kappa_2) $ compatible of level $ n $ and $ d_1<4 $ such that if
$ \tau\in \Pi^{\star}(n) $, $ T^{\tau} \subseteq T_2 $ and $
\Psi(\tau;0)=d_0,\ \Psi(\tau;1)=d_1 $ then no leaf $ \sigma $ of $
T^{\tau} $ has $\hat \Psi_0(\sigma;0)=\Psi(\tau;0) $ or $\hat
\Psi_1(\sigma;1)=\Psi(\tau;1) $. Proceeding inductively in this way
we are able to reach the required contradiction.

\vspace{4pt} Once again we define $ \Pi^{\star}(n) $ by recursion on
$ n $. We define $ \Pi^{\star}(0)=\{ \lambda \} $ and $ T^{\lambda}=
\{ \lambda \} $.
 Given $ \Pi^{\star}(n) $ we define $ \Pi^{\star}(n+1) $ as
follows.
 For each $ \tau \in \Pi^{\star}(n) $ let $
\{T_0,...,T_{m-1} \} $ be the set of all $ (T,2) $ compatible $ T' $
such that $ T^{\tau}\subset T' $ and which are of level $ n+1 $. Let
ncol$(n)=m'$ and let $ \tau_0,...,\tau_{mm'-1} $ be pairwise
incompatible extensions of $ \tau $ (these strings may be thought of
as being divided into $ m $ collections of size $ m' $), enumerate
these strings into $ \Pi^{\star}(n+1) $ and for each $ i< m $
proceed as follows: for each $ j $ with $ m'i\leq j<m'(i+1) $ define
$ T^{\tau_j}=T_i $ and $ \Psi(\tau_j;n)=j-m'i$.

\begin{lem} \label{easy2} For any $ n $, any $ f\in \omega^{<\omega} $ of length $ n $ and any $ (T,2) $ compatible $
T' $ of level $ n $, if it is the case that for all $ i<n $, $
f(i)<\mbox{ncol}(i) $ then there exists $ \tau\in \Pi^{\star}(n) $
such that $ T^{\tau}=T' $ and $ \Psi(\tau)=f $.
\end{lem}

\begin{proof} The proof is not difficult and is left to the reader.
\end{proof}

\vspace{4pt} We are now ready to define $ \Pi $. We do so in stages.

\vspace{4pt} \noindent \textbf{Stage 0}. We define $ \Pi_0= \{
\lambda \} $.

\vspace{4pt} \noindent \textbf{Stage} $\bf{s+1} $.  Initially $
\Pi_{s+1} $ is empty. For every string $ \tau $ which is a leaf of $
\Pi_s $ we proceed as follows. We shall have that $ \tau $ is a
string of level $ s $ in $ \Pi^{\star} $.  If it is not the case
that there exists $ \sigma \in T^{\tau} $ and $ i<s $ such that $
\hat \Psi_i(\sigma;i) $  is equal to $ \Psi(\tau;i) $ then enumerate
every successor of $ \tau $ in $ \Pi^{\star} $ into $ \Pi_{s+1} $,
together with all initial segments of such strings.

\vspace{4pt} We define $ \Pi= \bigcup_s \Pi_s $.

\begin{lem} \label{manlem} The class $ [\Pi] $ is non-empty.
\end{lem}
\begin{proof}
Suppose towards a contradiction that  $ n $ is the least such that
there exist no strings of level $ n $ in $ \Pi^{\star} $ which are
in $ \Pi $ (it follows from K\"{o}nig's lemma that such a
contradiction suffices to give the result). Let $ T_0
=\bigcup_{i\leq n}T(i) $. We proceed inductively to define $ T_i $
and $ d_{i-1} $ for each $ i\leq n, i\geq 1 $. Given $ T_i $ which
is $ (T,\kappa_i) $ compatible of level $ n $ we let the values $
\hat\Psi_i(\sigma;i) $ for those $ \sigma $ which are leaves of $
T_i $ define an ncol$(i) $-colouring of $ T_i $ (we may assume that
for any $ \sigma $ and any $ i $, if $\Psi_i(\sigma;i) $ is defined
then it is less than  $ \mbox{ncol}(i) $---otherwise we may regard
this computation as non-convergent). We then apply lemma \ref{nice}
in order to find $ T_{i+1}\subseteq T_{i} $ which is $
(T,\kappa_{i+1}) $ compatible of level $ n $ and $ d_i<$ncol$(i) $
such that no leaf $ \sigma $ of $ T_{i+1} $ has col$(\sigma)=d_i$.
Let the string $ f $ of length $ n $ be defined such that for all $
i<n $, $ f(i)=d_i $. By lemma \ref{easy2} there exists $ \tau\in
\Pi^{\star}(n) $ such that $ T^{\tau}= T_n $ and $ \Psi(\tau)=f $.
Then $ \tau $ is an element of $ \Pi^{\star}(n) $ which is in $ \Pi
$.

\end{proof}

\begin{lem} If $ A\in [\Pi] $ then $ deg(A) $ satisfies the cupping
property.
\end{lem}
\begin{proof} Suppose that $ A\in [\Pi] $ and that there exists $ C\in [T^A] $ such that $
\Psi_i(C)=\Psi(A) $. Then, in particular, there exists $ \sigma $ in
$ T^A $ which is an initial segment of $ C $ and such that $ \hat
\Psi_i(\sigma;i)=\Psi(A;i) $. Let $ \sigma $ be the shortest,
suppose that $ \sigma $ is of level $ n $ in $ T^A $ and let $ s $
be the least stage such that $ s>i+1$ and $ s>n $. At stage $ s $ in
the construction of $ \Pi $ we shall ensure that $ A\notin [\Pi] $.
This gives us the required contradiction.
\end{proof}

\vspace{4pt} What can we say about $ \Pi^0_1 $ classes every member
of which is of degree with strong minimal cover? Of course, it is a
trivial matter to define a $ \Pi^0_1 $ class which contains a single
(computable) member and so which satisfies this condition. On the
other hand, there cannot exist such a class of positive measure
since every such class contains a member of every random degree, and
therefore an element of every degree above $ \bf{0}' $. In a sense,
then, the following theorem is the strongest we could hope for:

\begin{theo}
There exists a non-empty $ \Pi^0_1 $ class with no computable
members, every member of which is of degree with strong minimal
cover.
\end{theo}

\begin{proof} We shall construct downwards closed $ \Pi $ such that
$ [\Pi] $ is non-empty and contains no computable elements, and such
that every member of this class is c.e. traceable. Let $ p $ be a
computable function which dominates every function $ p_i $ such that
$ p_i(n)=2^{n+i}(n+i)!$---why it that we consider this particular
function will become clear subsequently. We shall act in order to
satisfy the requirements:

\begin{enumerate}

\item[$\mathcal{C}_i:$]  $ (A\in [\Pi]) \wedge (\Psi_i(A) $ is
total) $ \Rightarrow $  there exists computable $ h $ such that $
\vline W_{h(n)} \vline \leq p(n) $ and $ \Psi_i(A;n)\in W_{h(n)} $
for all $ n $.

\item[$\mathcal{P}_i:$] If $ A\in [\Pi] $ then $ A\neq \Psi_i(\emptyset) $.

\end{enumerate}

\vspace{4pt} So let us consider first how to satisfy the requirement
$ \mathcal{C}_0 $. At the end of each stage $ s+1 $ we will add each
string $ \tau $ of length $ s+1 $ into $ \Pi $ unless $ \tau $ has
already been declared terminal. Various strings in $ \Pi $ will be
declared as `nodes' and will be allocated modules of two kinds. A $
\mathcal{C} $ module of the form $ (0,n) $  is concerned with
satisfaction of the requirement $ \mathcal{C}_0 $. If allocated to
the string $ \tau $ it searches at each stage $ s+1 $ for $ \tau'
\supseteq \tau $ with two incompatible extensions of length $ s $
which have not been declared terminal, and such that the computation
$ \Psi_0(\tau';n) $ converges in at most $ s $ steps. When the
module $ (0,n) $ finds such $ \tau' $  we say that the module
`acts'. All strings which are proper extensions of $ \tau $ are
declared not to be nodes and all modules are removed from these
strings. It chooses incompatible $ \tau_0, \tau_1 $ extending $
\tau' $ which are of length $ s $ and which have not been declared
terminal. It declares all extensions of $ \tau $ which are not
compatible with either $ \tau_i $ to be terminal and allocates the
module $ (0,n+1) $ to each of the $ \tau_i $ which are now declared
to be nodes. Thus if $ A\in [\Pi] $ extends $ \tau $ the module $
(0,n) $ restricts $ \Psi_0(A;n) $ to at most one possible value. The
module acts only once (for given $ \tau $). The modules $ (0,n+1) $
allocated to each $ \tau_i $ then combine in order to restrict $
\Psi_0(A;n+1) $ to at most two possible values, and so on.

\vspace{4pt} In order to satisfy the requirement $ \mathcal{C}_1 $
we shall proceed in just the same way except that, whereas we begin
to allocate modules for the sake of requirement $ \mathcal{C}_0 $ at
the node $ \lambda $, now we only begin to allocate modules for the
sake of this requirement at nodes of level 1 (nodes which have
precisely one proper initial segment as a node). In general we begin
to allocate modules for the sake of satisfying the requirement $
\mathcal{C}_i $ at nodes of level $ i $.

\vspace{4pt}
 The $ \mathcal{P} $ module $ (i) $
allocated to the node $ \tau $  is concerned with satisfaction of
the requirement $ \mathcal{P}_i $. We shall have that $ \tau $ is a
node of level $ i $. If $ \tau'\supset \tau $ is a node and there
does not exist any node $ \tau'' $ with $ \tau \subset \tau''
\subset \tau' $, then $ \tau' $ is called a `successor node' of $
\tau $. If $ (i) $ has not already acted then $ \tau $ will have two
successor nodes, $ \tau_0 $ and $ \tau_1 $ say.  If it finds at any
stage $ s+1 $ that $ \Psi_i(\emptyset) $ extends one of these
successor nodes $ \tau_i $ then the module `acts' (just the once) by
declaring all extensions of $ \tau $ which are not compatible with $
\tau_{\vline 1-i \vline} $ as terminal and not to be nodes, and by
removing all modules from these strings.

\vspace{4pt} The only point of any difficulty in this construction
is that the following kind of injury may occur. Let's suppose that $
\tau'\supset \tau $ are both nodes. A module allocated to $ \tau' $
may act so as to restrict the number of possible values $
\Psi_i(A;n) $ for all $ A\in [\Pi] $ extending $ \tau' $ and it may
then be the case that a module allocated to $ \tau $ acts and
defines a node or nodes of the form $ \tau'' $ such that $
\Psi_i(\tau'';n)\uparrow $. The simple remedy to this apparent
problem is to observe that we can easy bound the number of injuries
that can occur.

\vspace{4pt} We are ready to define the construction.

\vspace{4pt} \noindent \textbf{The module} $ (i,n) $
\textbf{allocated to} $ \tau $. At stage $ s+1 $ the module searches
for $ \tau' \supseteq \tau $ with two incompatible extensions of
length $ s $ which have not been declared terminal, and such that
the computation $ \Psi_i(\tau';n) $ converges in at most $ s $
steps. If there exists such $ \tau' $  we say that the module
`acts'. In this case it carries out the following instructions:
\begin{itemize}

\item  All strings which are proper extensions of $ \tau $ are
declared not to be nodes and all modules are removed from these
strings.

\item  It chooses incompatible $ \tau_0, \tau_1 $ extending $ \tau' $
which are of length $ s $ and which have not been declared terminal.

\item
The node $ \tau $ will have been allocated a module $ (j-1) $, here
$ j-1$ will be the level of $ \tau $ as a node. The module $ (i,n) $
then declares all extensions of $ \tau $ which are not compatible
with either $ \tau_i $ to be terminal and allocates the modules $
(j) $ and $ (j',j-j') $ for each $ j'\leq j $ to each of the $
\tau_i $ which are now declared to be nodes.

\item The module enumerates the tuple $ (i,n,\Psi_i(\tau';n))
$.
\end{itemize}

\vspace{4pt} The instructions for the module $ (i) $ allocated to $
\tau $ are precisely as previously described.

\vspace{4pt} \noindent \textbf{Stage 0}. Enumerate $ \lambda $ into
$ \Pi $, declare $ \lambda $ to be a node and allocate it the
modules $ (0,0) $ and $ (0) $.

\vspace{4pt} \noindent \textbf{Stage} $\bf{s+1}$. Run all modules
allocated to all nodes prior to this stage, in order of the level of
the node. All modules allocated to each node can be run in any fixed
order, but we only allow one module allocated to each node to act at
any given stage. For each string of $ \tau $ of length $ s+1 $ which
has not been declared terminal proceed as follows 1) enumerate $
\tau $ into $ \Pi $, 2) declare $ \tau $ to be a node, 3) if $ \tau
$ is a node of level $ n $ then allocate the modules $ (n) $ and $
(n',n-n') $ for every $ n' \leq n $ to $ \tau $.

\vspace{4pt} \noindent \textbf{The verification}.  Let's say that $
\tau $ is a  `final node' if there is a point of the construction at
which it is declared to be a node and after which it never
subsequently declared not to be a node. If $ \tau'\supset \tau $ are
both final nodes and there does not exist any final node $ \tau'' $
with $ \tau \subset \tau'' \subset \tau' $, then $ \tau' $ is called
a `final successor node' of $ \tau $. We say that $ \tau $ is a
final node of level $ n $ if it is a final node and has precisely $
n $ proper initial segments which are final nodes. We first show
that if $ \tau $ is a final node of level $ n $ then:

\begin{enumerate}
\item $ \tau $ has at at least one final successor node.

\item if $ \tau' $ is a final successor node of $ \tau $ then $
\tau' \not \subset \Psi_n(\emptyset) $.

\item if $ A\in [\Pi] $ extends $ \tau $ then it extends a final
successor node of $ \tau $.

\end{enumerate}
This suffices to show that $ [\Pi] $ is non-empty and that any $
A\in [\Pi] $ is incomputable. So suppose that $ \tau $ is a final
node of level $ n $. Then subsequent to the last stage  in the
construction at which $ \tau $ is declared to be a node, stage $ s $
say, no $ \mathcal{C} $ module allocated to a node which is a proper
initial segment of $ \tau $ acts. Let $ s'>s $ be the last stage at
which any $ \mathcal{C} $ module allocated to $ \tau $ acts, and if
there exists no such then let $ s'=s+1$. By the end of stage $ s' $,
$ \tau $ has precisely two successor nodes, $ \tau_0 $ and $ \tau_1
$ say and these are the only strings in $ \Pi $ extending $ \tau $
of length $ \vline \tau_0 \vline $ (we can assume that a $
\mathcal{P} $ module allocated to a node does not act until two
stages after the node has been declared). These two nodes are both
final and satisfy the property that they are not initial segments of
$ \Psi_n(\emptyset) $ unless the module $ (n) $ allocated to $ \tau
$ subsequently acts so as to declare some $ \tau_i $ terminal. In
this case the remaining successor node satisfies the required
property.

\vspace{4pt} We are left to show that if $ A\in [\Pi] $ then $ A $
is c.e. traceable. Consider the requirement $ \mathcal{C}_i $.
Observe first that for all $ n $, $ A $ extends a final node which
is allocated the module $ (i,n) $ so that if $ \Psi_i(A;n) $ is not
equal to some value $ d $ for which we enumerate a tuple $ (i,n,d) $
then $ \Psi_i(A;n) $ is undefined (it is a basic result of $ \Pi^0_1
$ classes that since $ [\Pi] $ contains no computable members it
must be perfect, so that if $ \Psi_i(\tau;n)\downarrow $ for some $
\tau\subset A $ then we will eventually be able to find the
incompatible extensions required for the module to act). Now if a
tuple of this form is enumerated then it is enumerated by a node of
level $ n+i $ and each such node can only enumerate at most one
tuple of this form---each time a string is declared to be a node we
consider this to be a new node. It therefore suffices to show that
for every $ n $ there exist at most $ 2^{n}(n+1)! $ many nodes
declared of level $ n $. If we do this then we will have shown that
for every $ n $ there at most $ 2^{n+i}(n+i+1)! $ values $ d $ for
which we enumerate a tuple $ (i,n,d) $. If we define $ W_{h'(n)} $
to be this set of values then for almost all $ n $ we have $ \vline
W_{h'(n)} \vline \leq p(n) $ so that a finite adjustment to $ h' $
yields the function $ h $ required for satisfaction of the
requirement $ \mathcal{C}_i $. We proceed by induction. The case $
n=0 $ is clear, so suppose that the result holds for all $ n'\leq n
$ so that there are at most $ 2^{n}(n+1)! $ nodes defined of level $
n $. Each such node is allocated $ n+1 $ different $ \mathcal{C} $
modules and can therefore have at most $ 2(n+2)$  different
successors during the construction. A simple calculation $
2(n+2).2^{n}(n+1)!=2^{n+1}(n+2)! $ completes the induction step.

\end{proof}

\section{The FPF degrees}

\begin{defin} We say that $ A\subseteq \omega $ is of fixed
point free (FPF) degree if there exists $ f\leq_T A $ such that $
\Psi_n(\emptyset) \neq \Psi_{f(n)}(\emptyset) $ for all $ n $. We
say $ f $ is DNR if for all $ n $ if $ \Psi_n(\emptyset;n)\downarrow
$ then $f(n)\neq \Psi_n(\emptyset;n) $.
\end{defin}

\begin{defin}
 We say that $ T \subseteq 2^{<\omega} $ is a \textbf{weak c.e.
tree} if it has a computable enumeration $ \{ T_s \}_{s\geq 0} $
such that $ T_0 =\{ \lambda \} $ and such that for all $ s\geq 0 $,
$ \vline T_{s+1} \vline -\vline T_s \vline \leq 1 $ and if $ \tau $
is in $ T_{s+1}-T_s $ then $ \tau $ is a leaf of $ T_{s+1} $.
\end{defin}

\begin{defin} We say a set of strings is \textbf{prefix-free} if its elements are
pairwise incompatible. Given $ T\subseteq 2^{<\omega} $  and $ \tau
\in T $ we define $ T_{\tau}= \{ \tau'\in T: \tau' \supseteq \tau
\}$ and we define $ \vline \tau \vline_T $ to be the level of $ \tau
$ in $ T $. We say that $ T'\subseteq T $ is $ T $-\textbf{thin} if
$ \lambda \in T' $ and for every $ \tau\in T' $ and every
prefix-free $ \Lambda \subseteq T'_{\tau} $ we have that $
\Sigma_{\tau'\in \Lambda} 2^{-\vline \tau' \vline_{T_{\tau}}}\leq 1
$.
\end{defin}

\begin{defin} We let $ C(\sigma) $ denote the plain Kolmogorov complexity
of $ \sigma $.
\end{defin}

\begin{theo}[\cite{KMS}] \label{KMS1} If $ A $ is of FPF degree then there exists $ g\leq_T A $ such that
for all $ n $, $ C(A\upharpoonright g(n))>n $.
\end{theo}

\begin{theo} \label{equiv} The following conditions on $ A \subseteq \omega $ are equivalent:

\begin{enumerate}
\item   For any weak c.e. tree $ T $, if $
A\in [T] $ then there exists c.e. $ T'\subseteq 2^{<\omega} $ which
is $ T $-thin and such that $ A\in [T'] $ ($ T' $ is not required to
be a weak c.e. tree).

\item There exists a computable   $ p
$ such that for every $ f\leq_T A $ there exists a computable $ h $
such that $ \vline W_{h(n)}\vline \leq p(n) $ for all $ n\in \omega
$ and for infinitely many $ n $ we have $ f(n)\in W_{h(n)} $.

\item For every  $ f\leq_T A $ there exists a
computable $ h $ such that $ \vline W_{h(n)}\vline \leq n $ for all
$ n\in \omega $ and for infinitely many $ n $ we have $ f(n)\in
W_{h(n)} $.

\item The degree of $ A $ is not FPF.

\end{enumerate}
\end{theo}

\begin{proof} (1)$ \Rightarrow $(2)
Given $ A $ which satisfies (1) and $ f=\Psi(A) $ let $ T $ be
defined as follows. The string of level 0 in $ T $ is the empty
string and for every $ n>0 $ the strings of level $ n $ in $ T $ are
those strings $ \tau $ such that $ \hat \Psi(\tau) $ is of length $
n $ and such that this is not the case for any $ \tau'\subset \tau
$. If $ T' $ is c.e. and $ T $-thin with $ A\in  [T'] $ then for
every $ n $ there exist at most $ 2^n $ strings in $ T' $ of level $
n $ in $ T $. Define $ W_{h(n)} $ to be the set of values $
\Psi(\tau)(n) $ for those $ \tau $ in $ T' $ of level $ n+1 $ in $ T
$ and define $ p(n)=2^{n+1} $. Note that $ p $ does not depend upon
$ \Psi $.

(2)$ \Rightarrow $(3) We proceed just as in \cite{TZ} (where the
same argument was made concerning a strengthening of the condition
of c.e. traceability). So suppose that $ A $ satisfies (2) and that
we are given $ f\leq_T A $. We can assume that $ p(0)=0 $ and that
for all $ n $, $ p(n+1)>p(n) $. For every $ n $ let $ k(n) $ be the
greatest $ m $ such that $ p(m)\leq n $. For every $ n $ let $ k'(n)
$ be the least $ m $ such that $ k(m)>n $. Define $ f'(n) $ to be an
effective coding of $ f \upharpoonright k'(n) $ and let $ h $ be
such as to satisfy condition (2) with respect to $ f' $ and $ p $.
Defining $ W_{h'(n)} $ to be the set of values $ \tau(n) $ for those
$ \tau $ whose codes are in $ W_{h(k(n))} $ suffices to show that
(3) is satisfied with respect to $ f $.

(3)$ \Rightarrow $(1) Suppose that $ A $ which satisfies (3) lies on
some weak c.e. tree $ T $. First we define $ T^{\star} $ as follows;
for every $ n $ the strings of level $ n $ in $ T^{\star} $ are the
strings of level $ \Sigma_{i\leq n}2i $ in $ T $. For all $ n $
define $ f(n) $ to be (some effective coding of) the initial segment
of $ A $ which is of level $ n $ in $ T^{\star} $. Let $ h $ be such
as to satisfy (3) with respect to $ f $. We can assume that if $
m\in W_{h(n)} $ then $ m $ codes a string of level $ n$ in $
T^{\star} $. Then $ T' $ which is the empty string together with all
strings whose codes are in  $ \bigcup_n W_{h(n)} $ is $ T $-thin
with $ A\in [T'] $. In order to see this suppose that $ \tau \in T'
$ and let $ \Lambda $ be the strings in $ T' $ which properly extend
$ \tau $. We show that $ \Sigma_{\tau'\in \Lambda}2^{-\vline \tau'
\vline_{T_{\tau}}} < 1 $. Suppose $ \tau $ is of level $ n $ in $
T^{\star} $. For every $ i>0 $ there are at most $ n+i $ strings in
$ T' $ which are of level $ n+i $ in $ T^{\star} $ (and extend $
\tau $) and each such string is of level at least $ 2(n+i) $ in $
T_{\tau} $. Then $ \Sigma_{\tau'\in \Lambda}2^{-\vline \tau'
\vline_{T_{\tau}}} \leq \Sigma_{i>0} (n+i)2^{-2(n+i)} < 1$.

(4)$ \Rightarrow $(3) It is well known that $ A $ is of FPF degree
iff $ A $ computes a DNR function, so if $ A $ is not of FPF degree
then for any $ f\leq_T A$ there exist infinitely many $ n $ with $
\Psi_n(\emptyset;n)\downarrow =f(n) $. For all $ n $ we can
therefore define $ W_{h(n)}= \{ \Psi_n(\emptyset;n) \} $ if $
\Psi_n(\emptyset;n)\downarrow $ and $ W_{h(n)} =\emptyset $
otherwise.

(3)$ \Rightarrow $(4) We suppose we are given $ A $ which is of FPF
degree and which satisfies (3) and then produce a contradiction. In
order to do so we extend an argument provided in \cite{KMS}. By
theorem \ref{KMS1} we may let  $ g\leq_T A $ be such that for all $
n $, $ C(A\upharpoonright g(n))>n $. For all $ n $ define $ f(n) $
to be some effective coding of $ A\upharpoonright g(n) $ and let $ h
$ be witness to the fact that (3) is satisfied with respect to $ f
$. There exists $ c $ such that, for all $ n $ with $ f(n)\in
W_{h(n)} $ we have $ C(A\upharpoonright g(n))\leq 3\ln n\ +c $ which
gives the required contradiction. In order to see this observe that
in order to specify $ f(n) $ (and so $ A\upharpoonright g(n)$)
whenever $ f(n)\in W_{h(n)} $ all we need is a string of the form $
\tau_0\tau_1 $ where in order to form $ \tau_0 $ we write $ n $ in
binary notation and then put a 0 after each bit except the last
after which we put a 1 (so that one can see where the coding of $ n
$ finishes and the coding of the position of $ f(n) $ within $
W_{h(n)} $ starts), and where in order to form $ \tau_1 $ we just
write $ m $ in binary notation where $ f(n) $ is the $ m^{th} $
element enumerated into $ W_{h(n)} $.

\end{proof}

\begin{theo} \label{favthe} Every $ \bf{0} $-dominated degree which is not FPF has
a strong minimal cover.
\end{theo}
\begin{proof} See sections 5 and 6.
\end{proof}

Of course, the question as to whether or not there exists a minimal
degree which is FPF was another longstanding question concerning
minimal degrees and it is interesting that, at least where the $
\bf{0} $-dominated degrees are concerned, these two questions now
seem to be related. In an unpublished paper \cite{MK}  Kumabe has
constructed a FPF minimal degree.

\begin{theo} \label{splittree}
If $ A $ satisfies the splitting-tree property then $ A $ is not of
FPF degree.
\end{theo}
\begin{proof} Suppose that $ A $ satisfies the splitting-tree property and
for some weak c.e. tree $ T $ we have that $ A\in [T] $. We define $
\Psi $ by enumerating axioms as follows: for every string $ \tau $
of level $ n $ in $ T $ we enumerate the axiom $
\Psi(\tau)=\tau\upharpoonright n $. Let $ T' $ be a c.e. $ \Psi
$-splitting tree such that $ A\in [T'] $. We can assume that $
T'\subseteq T $ and $ \lambda \in T' $. Then $ T' $ is c.e. and $ T
$-thin.
\end{proof}

Let us consider for a moment exactly what theorem \ref{splittree}
means. It is certainly the case that we may interpret this theorem
in a \emph{constructive} sense. Theorem \ref{splittree} tells us,
for example, that
 there exist $ \bf{0} $-dominated degrees which are not c.e. traceable
 and not FPF. This can be seen through an analysis of Gabay's proof
 that there exists a minimal degree with strong minimal cover and
 which is not c.e. traceable. The techniques developed suffice
 to give a set of minimal degree which is not c.e. traceable and which satisfies
  the perfect splitting tree property. Comments made
 in the introduction to this paper then suffice to show that this
 degree is $ \bf{0} $-dominated and theorem
 \ref{splittree} suffices to show that it is not FPF.
 One might also try to interpret theorem \ref{splittree}, though, as saying
 that the standard splitting tree technique cannot be used in order
 to construct a minimal degree which is  FPF, and that in order to
 do  so the use of delayed splitting trees is necessary. Of course, the
 functional $ \Psi $ defined in the proof of the theorem is so
 trivial that this case has not yet been made.
If it is the case that \emph{whenever} $ A\leq_T \Psi(A) $, $ A $
lies on a c.e. $ \Psi $-splitting tree then $ A $ is not of FPF
degree, but one might suppose that it is possible to proceed using
standard splitting trees while ignoring certain $ \Psi $ when for
some reason it will  obviously not be problematic to do so. It is
interesting to observe, anyway, that in constructing a minimal
degree which is FPF, Kumabe uses delayed splitting c.e. trees. In
light of theorem \ref{favthe} it seems reasonable to suggest that
delayed splitting c.e. trees are likely to be necessary in the
construction of a minimal degree with no strong minimal cover---or,
at least, in constructing such a degree which is also $ \bf{0}
$-dominated.

 \vspace{4pt} Since no 1-generic is FPF it follows from
theorem \ref{favthe} that any $ \bf{0} $-dominated degree which is
bounded by a 1-generic has strong minimal cover.  It therefore seems
of relevance to know whether there exist non-trivial examples of
such degrees. The following definition is due to Chong and Downey.

\begin{defin}  $ T\subseteq 2^{<\omega}  $   is said
to be $ \Sigma_1 $ \textbf{dense} in $ A $ if:

\begin{itemize}
\item no element of $ T $ is an initial segment of $ A $,

\item for any c.e. $ T'\subseteq 2^{<\omega}  $ such that $ A \in [D(T')] $, some
member of  $ T' $ extends a member of $ T $.
\end{itemize}
\end{defin}

Chong and Downey \cite{CD1}, \cite{CD2} have shown that any set $ A
$ is computable in a 1-generic iff there is no c.e. set of strings $
T $ which is $ \Sigma_1 $ dense in $ A $.

\vspace{4pt} Using this characterization they were able to show that
there is a minimal degree below $ \bf{0}'$ which is  bounded by a
1-generic below $ \bf{0}'' $, and also that there is a minimal
degree below $ \bf{0}' $ which is not bounded by any 1-generic. The
following theorem has also been proved independently in a joint
paper by
 Downey and Yu \cite{DY}.

\begin{theo} There are
$ \bf{0} $-dominated degrees which are not bounded by any 1-generic
and $ \bf{0} $-dominated degrees ($\neq \bf{0} $) which are bounded
by a 1-generic degree.
\end{theo}
\begin{proof}   That there
exist $ \bf{0} $-dominated degrees which are not bounded by any
1-generic follows from the fact that there exist $ \bf{0}
$-dominated degrees which are FPF.  In order to show that there
exists a $ \bf{0} $-dominated degree ($\neq \bf{0} $) which is
bounded by a 1-generic we may proceed almost exactly as in
\cite{CD1} in order to construct $ A $ of $ \bf{0} $-dominated
minimal degree such that there is no c.e. set of strings $ T $ which
is $ \Sigma_1 $ dense in $ A $.  We construct a set $ A $ of minimal
degree below $ \bf{0}'' $ which lies on perfect splitting trees. At
each stage $ s+1 $, having defined $\tau \supseteq  A_s $ of which $
A_{s+1} $ will be an extension and a tree $ T_{s+1} $ which
satisfies the property that if $ A\in [T_{s+1}] $ then the $ s^{th}
$ minimality requirement will be satisfied, we then ask whether
there exists $ \tau' \in \Pi_s $---the $ s^{th} $ c.e. set of
strings---which is extended by a string in $ T_{s+1} $ extending $
\tau $. If so then we may define $ A_{s+1} $ so as to extend such $
\tau' $ and otherwise the strings in $ T_{s+1} $ extending $ \tau $
are a c.e. set of strings which is witness to the fact that $ \Pi_s
$ is not $ \Sigma_1 $ dense in $ A $.
\end{proof}

We close this section by observing that another technique of minimal
degree construction always produces minimal degrees with a strong
minimal cover.

\begin{defin}
For any $ \Pi\subseteq 2^{<\omega} $ we define $ B([\Pi]) $, the
\textbf{Cantor-Bendixson derivative of $ [\Pi] $}, to be the set of
non-isolated points of $ [\Pi] $ according to the Cantor topology.
The \textbf{ iterated Cantor-Bendixson derivative $
B^{\alpha}([\Pi]) $} is defined for all ordinals $ \alpha $ by the
following transfinite induction. $ B^0([\Pi])=[\Pi] $, $ B^{\alpha
+1}([\Pi])=B(B^{\alpha}([\Pi])) $ and $ B^{\lambda}([\Pi]) =
\bigcap_{\alpha<\lambda} B^{\alpha}([\Pi]) $ for any limit ordinal $
\lambda $.
\end{defin}

\begin{defin}
A set $ A $ has \textbf{Cantor-Bendixson rank $ \alpha $} if $
\alpha $ is the least ordinal such that for some $ \Pi_1^0 $ class $
[\Pi] $, $ A\in B^{\alpha}([\Pi])-B^{\alpha+1}([\Pi]) $.
\end{defin}

\begin{theo}[Cenzer, Smith \cite{CS}] \label{CS}  If $ B\leq_{tt} A $ and $ A $ has rank $
n  $ then $ B $ has rank $ m\leq n $.
\end{theo}

\begin{theo}[Owings \cite{JO}] \label{Ow}  If $ rk(B)=rk(A\oplus B) $ then $ A\leq_T B $.
\end{theo}

\begin{theo}[Downey \cite{RD}] \label{BS}  There exists a set of $ \bf{0} $-dominated degree
and which is of rank one.
\end{theo}

\begin{theo} If $ A $ is of rank one and is of $ \bf{0} $-dominated
degree then it is of minimal degree.
 \end{theo}
  \begin{proof}
Suppose that $ A $ is of $ \bf{0} $-dominated degree, that $ A $ is
of rank 1, and that there exists incomputable $ B<_T A $. Generally
speaking whenever $ C\leq_T D $ and $ D $ is of $ \bf{0} $-dominated
degree we actually have $ C\leq_{tt} D $. By theorem \ref{CS}, then,
$ B $ must be of rank 1 since to be of rank 0 would mean that $ B $
is computable. But then $ A\oplus B $ is also of rank 1 so that by
theorem \ref{Ow} we have that $ A $ is computable in $ B $ which
gives a contradiction.
\end{proof}

\begin{theo} \label{smcrank} If $ A $ is of $ \bf{0} $-dominated degree and is of
rank 1 then the degree of $ A $ is not FPF and therefore has strong
minimal cover.
\end{theo}

\begin{proof} We suppose we are given $ A $ which satisfies the
hypothesis of the theorem and we show that this set satisfies (3) of
theorem \ref{equiv}. So let $ [\Pi]  $ be a $ \Pi^0_1 $ class such
that $ A $ is the unique non-isolated point of $ [\Pi] $. Given $
f=\Psi(A) $ we may take computable $ g $ which majorizes the use
function for $ \Psi $ with oracle $ A $. For every $ n $ we consider
$ \Pi(n) $, the strings in $ \Pi $ of length $ n $. Suppose there
are $ m $ strings in this set. Since only one of these strings is an
initial segment of a non-isolated point of $ [\Pi] $ there exists
some $ \Lambda \subset \Pi(n) $ of size $ m-1$ and some large $
n',n'' $ with $ n''>g(n') $, such that there do not exist $ n' $
strings in $ \Pi(g(n')) $ extending a string in $ \Lambda $ and
which have an extension in $ \Pi(n'') $. In fact we can effectively
find such $ \Lambda, n',n'' $ so that for each $ n $, having found
such values, we can define $ W_{h(n')} $ to be the set of values $
\Psi(\tau;n') $ for those $ \tau $ in $ \Pi(g(n')) $ extending a
string in $ \Lambda $ which have an extension in $ \Pi(n'') $ and we
can enumerate also the string, $ \tau_n $ say, which is the unique
element of $ \Pi(n)-\Lambda $. For each $ n $ (considered in turn)
we can insist that $ n' $ is larger than any number previously
mentioned during the construction and then define $ W_{h(k)} $, for
all $ k<n' $, to be the empty set unless this value is already
defined. If it is the case that for almost all $ n $, $ \tau_n
\subset A $ then $ A $ is computable and otherwise we have that for
infinitely many $ n $, $ f(n)\in W_{h(n)} $.
\end{proof}

\section{Constructing a strong minimal cover}

In this section we shall discuss a straightforward approach to be
taken in attempting to construct a strong minimal cover for any
given degree. In so doing lemma \ref{trelem1} will be useful.

\begin{defin}
We shall say that $ \tau $ is $ \textbf{A} \oplus
$-\textbf{compatible} if, for all $ n $ such that $
\tau(2n)\downarrow $, we have $ \tau(2n)=A(n) $.
\end{defin}

\begin{lem} \label{trelem1} Suppose $ \Psi =\hat \Psi $.
If $ T_0  $ is an $ A $-computable 2-branching $ \Psi $-splitting
tree, then $  T_1= \{ \Psi(\tau):\tau\in T_0 \} $
 is an $ A $-computable 2-branching tree.
Let $ T_2 $ be an $ A $-computable 2-branching subtree of $ T_1 $.
Then $ T_3= \{ \tau \in T_0: \Psi(\tau)\in T_2 \} $
 is an $ A $-computable 2-branching subtree of $ T_0 $.
\end{lem}

\begin{proof} The proof is not difficult and is left to the reader.
\end{proof}

So now let us suppose that we wish to construct a strong minimal
cover for $ deg(A) $. In order to do so we must construct $ B \geq_T
A $ and satisfy all requirements;

\vspace{4pt}

$ \mathcal{R}_{i}: \Psi_i(B) \ \mbox{total} \ \rightarrow (\Psi_i(B)
\leq_T A \ \ \mbox{or} \ \ B \leq_T \Psi_i(B)) $

$ \mathcal{P}_{i}: B \neq \Psi_i(A) $

\vspace{4pt} In order to ensure that $ B \geq_T A $ we can simply
insist that $ B $ should be an $ A\oplus $-compatible string. Thus
we begin with the restriction that $ B $ should lie on the tree $
T_0 $ containing all strings of even length which are $ A \oplus
$-compatible, an $ A $-computable 2-branching
 tree.

In order to meet all other requirements we might try to proceed by
finite extension. We define $ B_0 $ to be the empty string. Suppose
that by the end of stage $ s $ we have defined $ B_s \in 2^{<\omega
} $ on $ T_s $,
 an $ A $-computable 2-branching tree, in such a way that if $ B $ extends $ B_s $ and lies on $ T_s $ then all
  requirements $ \mathcal{R}_i ,\mathcal{P}_i $ for $i<s $ will be satisfied.
At stage $ s+1 $ we might proceed, initially, just as if we were
only trying to construct a minimal cover for $ A $. We can assume
that $ \Psi_s = \hat \Psi_s $ (otherwise replace $ \Psi_s $ with $
\hat \Psi_s $ in what follows). We ask the question, ``does there
exist $ \tau \supseteq B_s $ on $ T_s $ such that no two strings on
$ T_s $ extending $ \tau $ are a $ \Psi_s $-splitting?".

\vspace{4pt}
 \textbf{If so}: then let $ \tau $ be such a string. We can define $ T_{s+1}=T_s $ and
 (just to make the satisfaction of $ \mathcal{P}_s $ explicit)
define $ B_{s+1} $ to be some extension of $ \tau $ on $ T_s $
sufficient to ensure $ \mathcal{P}_s $ is satisfied.

\vspace{4pt}
 \textbf{If not}: then we can define $ T_s' $ to be an $ A $-computable 2-branching $ \Psi_s $-splitting
  subtree of $ T_s $ having $ B_s $ as least element---the idea being that
we shall eventually define $ T_{s+1} $ to be some subtree of $ T_s'
$. If $ B $ lies on $ T_s' $ then we shall
 have that $ B \leq_T \Psi_s(B) \oplus A $. Of course this does not suffice, since
for the satisfaction of $ \mathcal{R}_s $ we require that $ B \leq_T
\Psi_s(B) $. Suppose, however, that we know
 $ A $ satisfies the property;

\vspace{4pt} \noindent $ (\dagger) $ If $ T $ is an $ A $-computable
2-branching tree, then there exists $ T' $ a subtree of $ T $ which
is also an
  $ A $-computable 2-branching tree and which satisfies the property that if $ C \in [T'] $ then $ A \leq_T C $.

\vspace{4pt} Lemma \ref{trelem1} then suffices to ensure that we can
define $ T_{s+1} $ to be a subtree of $ T_s' $
 which is an $ A $-computable 2-branching tree,
and which satisfies the property that if $ B \in [T_{s+1}] $ then $
A \leq_T \Psi_s(B) $ so that, since $ B \leq_T \Psi_s(B) \oplus A $,
$B \leq_T \Psi_s(B) $. Then we can define $ B_{s+1} $ to be an
extension of $ B_s $ lying on $ T_{s+1} $ sufficient to ensure the
satisfaction of $ \mathcal{P}_s $.

In conclusion, then, if $ A $ satisfies the property $ (\dagger) $
we can construct a strong minimal cover for $ deg(A) $.

\section{The proof of theorem \ref{favthe}}

The remarks of the last section suffice to show that in order to
prove theorem \ref{favthe} we need only prove that if $ A $ is of $
\bf{0} $-dominated degree which is not FPF then $ A $ satisfies $
(\dagger) $. So let us now assume that $ A $ is such a set and that
we are given a Turing functional $ \Phi $ which satisfies the
property that for all $ \sigma $, $ \Phi(A;\sigma)\downarrow \in \{
0,1 \} $ and $ \Phi(A;\sigma)=1 $ iff $ \sigma \in T $, where $ T $
 is  $ A $-computable and 2-branching. We can assume that $ T $
has a single element of level 0 which is the empty string and  that
$ \Phi(\tau;\sigma)\downarrow  $ only if the computation converges
in $ \leq \vline \tau \vline $ steps and $
\Phi(\tau;\sigma')\downarrow $ for all $ \sigma' $ such that $
\vline \sigma' \vline <\vline \sigma \vline $.

\begin{defin}
For all $ \tau $ we define  $ T(\tau) =\{ \sigma :
\Phi(\tau;\sigma)\downarrow =1 \} $.
\end{defin}

It will be convenient, also, to assume that for any $ \tau $ and $
\sigma \in T(\tau) $, $ \sigma $ has at most two successors in $
T(\tau) $ and that any string of level 0 in $ T(\tau) $ must be $
\lambda $.

Consider now the $ A $-computable function $g$ defined as follows;
for every $ n $, $ g(n) $ is the greatest value $ \vline \sigma
\vline $ such that $ \sigma $ is of level $ n $ in $ T $. Since  $ A
$ is of $ \bf{0} $-dominated degree we can take computable and
increasing $ f $ which majorizes $ g $.

\begin{defin}
We denote $ \Omega(\tau,n) $ iff $ n=0 $ or:
\begin{itemize}
\item  $ T(\tau) $ is of level at least $ n $ and is 2-branching below level $ n $, and

\item for every $ n'\leq n $, the greatest value $ \vline \sigma \vline $ such that
$ \sigma $ is of level $ n' $ in $ T(\tau) $ is $ \leq f(n') $.
\end{itemize}
\end{defin}

\subsection{Defining $ \Pi $}
We define $ \Pi $ by enumeration in stages. Initially $ \Pi $
contains only the empty string.
 At stage $ s>0 $ we consider all strings $ \tau $ of length $ s $ and for each such string we proceed as follows.
Let $ n $ be the greatest such that $ \Omega(\tau',n) $ for $ \tau' \subset \tau $ such that $ \tau' \in \Pi $.
 If it is the case that $ \Omega(\tau,n') $ for some $ n'\geq n+2 $
and all strings in $ T(\tau) $ of level $ n' $ are of length at
least $ f(n+2) $ then enumerate $ \tau $ into $ \Pi $.

\subsection{Using  $ \Pi $-thin $ \Pi' $. } Let $ \Pi' $ be $ \Pi $-thin and
let $ \Lambda = \{ \tau_i: 1\leq i \leq k \} $ be a prefix-free set
of strings in $ \Pi' $  such that each $ \tau_i $ extends $ \tau\in
\Pi' $.  For each $ i $ let $ n_{\tau_i} $ be the greatest $ n $
such that $ \Omega(\tau_i,n) $ and let $ n_{\tau} $ be the greatest
$ n $ such that $ \Omega(\tau,n) $. We show that if $ \sigma $ is
any string in $ T(\tau) $  of level $ n_{\tau} $ then we can choose
two strings extending $ \sigma $ of level $ n_{\tau_i } $ from each
$ T(\tau_i) $, $ \sigma_{i,0} $ and $ \sigma_{i,1} $ say, so that if
$ i\neq i' $ or $ j \neq j' $ then $ \sigma_{i,j} \vline
\sigma_{i',j'} $.

Define $ m_0 = \vline \tau \vline_{\Pi} $. For each $ m\geq 1 $ let
$ \Lambda_{m} $ be the set of $ \tau_{i} $ which are of level $ \leq
m_0+m $ in $ \Pi $ and let $ \Lambda^{\star}_{m} $ be the set of $
\tau_{i} $ which are of level $  m_0+m $ in $ \Pi $. Defining $ r_m
= \Sigma_{m'= 1}^{m} 2^{-m'} \vline \Lambda^{\star}_{m'} \vline $ we
have that $ r_m \leq 1 $. We show by induction on $ m $ that we can
choose two strings $ \sigma_{i,0} $ and $ \sigma_{i,1} $ extending $
\sigma $ of level $ n_{\tau_i } $ from each $ T(\tau_i) $ such that
$ \tau_i \in \Lambda_m $ and $ (1-r_m) 2^{m+1}  $ different strings
$ \psi_{i,j} $ extending $ \sigma $ of level $ n_{\tau}+2m $ from
each $ T(\tau_i) $ for $ \tau_i \in \Lambda - \Lambda_m $ in such a
way that (where these values are defined):

\begin{itemize}
\item if $ i \neq i' $ or $ j\neq j' $ then $ \sigma_{i,j} \vline \sigma_{i',j'} $, and

\item for any $ i,j,i',j' $ we have $ \sigma_{i,j} \vline \psi_{i',j'} $.
\end{itemize}

\noindent Case $ m=1 $. If $ \vline \Lambda^{\star}_1 \vline =0 $
then the result is clear, so suppose first that $ \vline
\Lambda^{\star}_1 \vline =1 $ and (simply for the sake of simplicity
of labeling) let us suppose that $ \tau_1 \in \Lambda^{\star}_1 $.
We choose any two different strings from $ T(\tau_1) $ of level
$n_{\tau_1} $ extending $ \sigma $ and define these to be $
\sigma_{1,0} $ and $ \sigma_{1,1} $. Observe that every string in $
T(\tau_1) $ of level $n_{\tau_1} $ is of length $ \geq f(n_{\tau}+2)
$. Now consider those $ \tau_i \in \Lambda- \Lambda^{\star}_1 $.
Since every string in  $ T(\tau_i) $ of level $ n_{\tau}+2 $ is of
length $ \leq f(n_{\tau}+2) $ there are at most two strings in $
T(\tau_i) $ of level $ n_{\tau}+2 $ which are compatible with either
$ \sigma_{1,0} $ or $ \sigma_{1,1} $.
 We can therefore define $ \psi_{i,0} $  and $ \psi_{i,1} $ as required.

Suppose $ \vline \Lambda^{\star}_1 \vline =2 $ and  $ \tau_1,\tau_2
\in \Lambda^{\star}_1 $. First we choose any two different strings
from $ T(\tau_1) $ of level $n_{\tau_1} $ extending $ \sigma $ and
define these to be $ \sigma_{1,0} $ and $ \sigma_{1,1} $. Since
every string in  $ T(\tau_2) $ of level $ n_{\tau}+2 $ is of length
$ \leq f(n_{\tau}+2) $ there are at most two strings in $ T(\tau_2)
$ of level $ n_{\tau}+2 $ which are compatible with either $
\sigma_{1,0} $ or $ \sigma_{1,1} $. We can therefore define $
\sigma_{2,0} $  and $ \sigma_{2,1} $ as required.

\begin{pgfpicture}{0mm}{-20mm}{114mm}{-79mm}

\color[rgb]{0,0,0}\pgfsetlinewidth{0.254mm}\pgfline{\pgfpoint{45.72mm}{-55.88mm}}{\pgfpoint{66.04mm}{-41.91mm}} 
        \color[rgb]{0,0,0}\pgfsetlinewidth{0.254mm}\pgfline{\pgfpoint{53.34mm}{-64.77mm}}{\pgfpoint{90.17mm}{-46.99mm}} 

        \color[rgb]{0,0,0}\pgfsetlinewidth{0.254mm}\pgfline{\pgfpoint{60.96mm}{-72.39mm}}{\pgfpoint{19.05mm}{-27.94mm}} 

        \color[rgb]{0,0,0}\pgfsetlinewidth{0.254mm}\pgfline{\pgfpoint{90.17mm}{-46.99mm}}{\pgfpoint{100.33mm}{-26.67mm}} 
        \color[rgb]{0,0,0}\pgfputat{\pgfpoint{0mm}{-34.29mm}}{\pgfbox[left,center]{\shortstack{$f(n_{\tau}+2)$}}} 

        \color[rgb]{0,0,0}\pgfputat{\pgfpoint{52.07mm}{-77.47mm}}{\pgfbox[left,center]{\shortstack{Figure 1}}} 
        \color[rgb]{0,0,0}\pgfsetlinewidth{0.254mm}\pgfline{\pgfpoint{36.83mm}{-46.99mm}}{\pgfpoint{43.18mm}{-26.67mm}} 
        \color[rgb]{0,0,0}\pgfsetlinewidth{0.254mm}\pgfline{\pgfpoint{66.04mm}{-41.91mm}}{\pgfpoint{68.58mm}{-21.59mm}} 
        \color[gray]{0.7}\pgfellipse[fill]{\pgfpoint{29.21mm}{-39.37mm}}{\pgfpoint{1.27mm}{0mm}}{\pgfpoint{0mm}{1.27mm}}\color[rgb]{0,0,0}\pgfsetlinewidth{0.254mm}\pgfellipse[stroke]{\pgfpoint{29.21mm}{-39.37mm}}{\pgfpoint{1.27mm}{0mm}}{\pgfpoint{0mm}{1.27mm}} 
        \color[gray]{0.7}\pgfellipse[fill]{\pgfpoint{39.37mm}{-38.1mm}}{\pgfpoint{1.27mm}{0mm}}{\pgfpoint{0mm}{1.27mm}}\color[rgb]{0,0,0}\pgfsetlinewidth{0.254mm}\pgfellipse[stroke]{\pgfpoint{39.37mm}{-38.1mm}}{\pgfpoint{1.27mm}{0mm}}{\pgfpoint{0mm}{1.27mm}} 
        \color[gray]{0.7}\pgfellipse[fill]{\pgfpoint{66.04mm}{-41.91mm}}{\pgfpoint{1.27mm}{0mm}}{\pgfpoint{0mm}{1.27mm}}\color[rgb]{0,0,0}\pgfsetlinewidth{0.254mm}\pgfellipse[stroke]{\pgfpoint{66.04mm}{-41.91mm}}{\pgfpoint{1.27mm}{0mm}}{\pgfpoint{0mm}{1.27mm}} 
        \color[gray]{0.7}\pgfellipse[fill]{\pgfpoint{90.17mm}{-46.99mm}}{\pgfpoint{1.27mm}{0mm}}{\pgfpoint{0mm}{1.27mm}}\color[rgb]{0,0,0}\pgfsetlinewidth{0.254mm}\pgfellipse[stroke]{\pgfpoint{90.17mm}{-46.99mm}}{\pgfpoint{1.27mm}{0mm}}{\pgfpoint{0mm}{1.27mm}} 
        \color[rgb]{0,0,0}\pgfputat{\pgfpoint{16.51mm}{-26.4mm}}{\pgfbox[left,center]{\shortstack{$\sigma_{1,0}$}}} 
        \color[rgb]{0,0,0}\pgfputat{\pgfpoint{39.37mm}{-25.4mm}}{\pgfbox[left,center]{\shortstack{$\sigma_{1,1}$}}} 
        \color[rgb]{0,0,0}\pgfputat{\pgfpoint{64.77mm}{-20.59mm}}{\pgfbox[left,center]{\shortstack{$\sigma_{2,0}$}}} 
        \color[rgb]{0,0,0}\pgfputat{\pgfpoint{96.25mm}{-25.4mm}}{\pgfbox[left,center]{\shortstack{$\sigma_{2,1}$}}} 
        \color[rgb]{0,0,0}\pgfputat{\pgfpoint{61.23mm}{-70.12mm}}{\pgfbox[left,center]{\shortstack{$\sigma$}}} 

\color[rgb]{0,0,0}\pgfsetlinewidth{0.254mm}\pgfsetdash{{1pt}{5pt}{1pt}{5pt}}{0mm}\pgfline{\pgfpoint{113.03mm}{-72.39mm}}{\pgfpoint{17.78mm}{-72.39mm}} 
\color[rgb]{0,0,0}\pgfsetlinewidth{0.254mm}\pgfsetdash{{1pt}{5pt}{1pt}{5pt}}{0mm}\pgfline{\pgfpoint{17.78mm}{-34.29mm}}{\pgfpoint{113.03mm}{-34.29mm}} 

        \color[rgb]{0,0,0} 
\end{pgfpicture}

\newpage

\noindent The diagram illustrates what happens in the case that $
\vline \Lambda^{\star}_1 \vline =2 $. First we pick $\sigma_{1,0}$
and $\sigma_{1,1}$. These are strings  from $ T(\tau_1) $ of level
$n_{\tau_1} $ extending $ \sigma $, and are therefore of length
$\geq f(n_{\tau}+2)$. The coloured circles indicate what the strings
extending $\sigma$ and of level $n_{\tau}+2$ may look like in
$T(\tau_2)$. These strings are of length $\leq f(n_{\tau}+2)$ and
therefore at most two of them are compatible with either of the
strings $ \sigma_{1,0}$, $\sigma_{1,1}$. We may therefore choose $
\sigma_{2,0}$ and $ \sigma_{2,1}$ of level $n_{\tau_2}$ in
$T(\tau_2)$, which are incompatible with $ \sigma_{1,0}$ and
$\sigma_{1,1}$. These strings will be of length $ \geq
f(n_{\tau}+2)$.

\vspace{4pt}
\noindent Case $ m>1 $. By the induction hypothesis we can choose  two strings $ \sigma_{i,0} $ and $ \sigma_{i,1} $
 extending $ \sigma $
of level $ n_{\tau_i } $ from each $ T(\tau_i) $ such that $ \tau_i
\in \Lambda_{m-1} $ and $ (1-r_{m-1}) 2^{m}  $ different strings $
\psi_{i,j} $ extending $ \sigma $ of level $ n_{\tau}+2(m-1) $ from
each $ T(\tau_i) $ for $ \tau_i \in \Lambda - \Lambda_{m-1} $ in
such a way that if $ i \neq i' $ or $ j\neq j' $ then $ \sigma_{i,j}
\vline \sigma_{i',j'} $ and for any $ i,j,i',j' $ we have $
\sigma_{i,j} \vline \psi_{i',j'} $. For each $ \tau_i \in \Lambda-
\Lambda_{m-1} $ take the four extensions of each $ \psi_{i,j} $ of
level $ n_{\tau}+2m $ in $ T(\tau_i) $ and relabel so that these are
the strings $ \psi_{i,j} $. There are at most $ (1-r_{m-1}) 2^{m}  $
strings in $ \Lambda^{\star}_{m} $. We proceed first by defining in
turn the strings $ \sigma_{i,j} $ such that $ \tau_i \in
\Lambda^{\star}_{m} $, from amongst the extensions of the strings $
\psi_{i,j'} $. Whenever we define such $ \sigma_{i,j} $ it is of
length $ \geq f(n_{\tau}+2m) $ and it is therefore the case that for
each $ \tau_{i'} \in \Lambda-\Lambda_{m-1} $ there is at most one
string $ \psi_{i',j'} $ which is compatible with $ \sigma_{i,j} $.
Since we must choose at most  $ (1-r_{m-1}) 2^{m+1} $ strings $
\sigma_{i,j} $ and each $ \tau_{i'} \in \Lambda-\Lambda_{m-1} $ has
$ (1-r_{m-1}) 2^{m+2}  $ incompatible $ \psi_{i',j'} $ we can define
all the $ \sigma_{i,j} $ as required. This
 leaves each $ \tau_i \in \Lambda - \Lambda_m $ with at least
$ (1-r_{m-1}) 2^{m+1} \geq (1-r_{m}) 2^{m+1} $ strings $ \psi_{i,j} $
 which are incompatible with any $ \sigma_{i',j'} $.

\subsection{Defining $ T' $.}
Let c.e. $ \Pi' $ be $ \Pi $-thin with $ A\in [\Pi'] $. Since $ A $
is of $ \bf{0} $-dominated degree we can let $ \Pi^{\star } $ be a
subset of $ \Pi' $ such that:

\begin{itemize}
\item  $ A \in [\Pi^{\star }] $, and $ \Pi^{\star } $ has as element of level 0 the empty string $ \lambda$,

\item  each $ \tau \in \Pi^{\star } $ has a finite number of
successors, and

\item  there is a computable function which given any $ \tau $
such that $\tau \in \Pi^{\star } $ returns $ m $ such that $ D_m $
(the $ m^{th} $ finite set according to some fixed effective
listing) codes the successors of $ \tau $ in $ \Pi^{\star } $.
\end{itemize}

 Given a computable enumeration $ \{
\Pi^{\star }_{s} \}_{s\geq 0} $
 satisfying
 \begin{enumerate}
 \item  $ \Pi^{\star}_0= \{ \lambda \} $,

 \item  if $ \tau \in \Pi^{\star}_{s+1}-\Pi^{\star}_s $
 then $ \tau $ extends a leaf of $ \Pi^{\star}_s $, and

 \item  if
 $ \tau,\tau' \in \Pi^{\star}_{s+1}-\Pi^{\star}_s $ then these strings are incompatible,

 \end{enumerate}

 \noindent we proceed in an effective fashion to enumerate values
  $ T'(\tau) $  and axioms for $ \Theta $ such that
$ T' = \bigcup \{ T'(\tau): T'(\tau)\downarrow, \tau \subset A \} $
is an $ A $-computable 2-branching subtree of $ T $, and for all $
C\in [T'] $ we have $ \Theta(C)=A $. This suffices, then, to show
that $ A $ satisfies $ (\dagger) $, as required.

\vspace{4pt} \noindent \textbf{Stage 0}. We define $ T'(\lambda) =
\{ \lambda \} $.

\vspace{4pt} \noindent \textbf{Stage s+1}. We can assume that
strings are enumerated into $ \Pi^{\star }_{s+1} $ extending
precisely one leaf of $ \Pi^{\star }_{s} $, $ \tau $ say, which is a
string of level $ m $ (say) in $\Pi^{\star} $.  Let the strings
enumerated into $ \Pi^{\star }_{s+1} $ extending $ \tau $ be $
\tau_1,..,\tau_{k} $.
 For each $ i $ let $ n_{\tau_i} $
be the greatest $ n $ such that $ \Omega(\tau_i,n) $ and let $
n_{\tau} $ be the greatest $ n $ such that $ \Omega(\tau,n) $. We
will have already defined the value $ T'(\tau) $ to be a tree of
level $ m $, which is 2-branching below level $ m $, and with each
leaf a string of level $ n_{\tau} $ in $ T(\tau) $.

Now we must define each $ T'(\tau_i) $ to be a tree of level $ m+1 $
which is 2-branching below level $ m+1 $,
 with $ T'(\tau) $
 as subtree and with two leaves extending each leaf $ \sigma $ of
$ T'(\tau) $, with each leaf a string of level $ n_{\tau_i} $ in $
T(\tau_i) $, and such that for $ i \neq i' $ any leaf of $
T'(\tau_i) $ is incompatible with any leaf of $ T'(\tau_{i'}) $.
Thus for each leaf $ \sigma $ of $ T'(\tau) $ we must choose for
each $ \tau_i $ two extensions $ \sigma_{i,0} $ and $ \sigma_{i,1} $
of level $ n_{\tau_i} $ in $ T(\tau_i) $ in such a way that $
\sigma_{i,j} \vline \sigma_{i',j'} $ if either $ i\neq i' $ or $ j
\neq j' $. The observation of section 6.2 says precisely that this
is possible.
 Since these strings are pairwise incompatible
we can then consistently define $ \Theta(\sigma')= \tau_i $, for
each $\sigma'$ which we have just defined as a leaf of $ T'(\tau_i)
$.

\bibliographystyle{alpha}

\section*{Addresses}
\alladdress

\end{document}